\documentclass[sn-mathphys,Numbered]{sn-jnl}% Math and Physical Sciences 
\usepackage{graphicx}%
\usepackage{multirow}%
\usepackage{amsmath,amssymb,amsfonts}%
\usepackage{amsthm}%
\usepackage{mathrsfs}%
\usepackage[title]{appendix}%
\usepackage{xcolor}%
\usepackage{textcomp}%
\usepackage{manyfoot}%
\usepackage{booktabs}%
\usepackage{algorithm}%
\usepackage{algorithmicx}%
\usepackage{algpseudocode}%
\usepackage{listings}%
\theoremstyle{thmstyleone}%
\newtheorem{theorem}{Theorem}%  meant for continuous numbers
\newtheorem{corollary}{Corollary}%  meant for continuous numbers
\newtheorem{proposition}{Proposition}% 

\theoremstyle{thmstyletwo}%
\newtheorem{remark}{Remark}%

\theoremstyle{thmstylethree}%
\newtheorem{definition}{Definition}%

\begin{document}
	
	\title[Hyperstability of some functional equations]{Hyperstability of some functional equations in modular spaces}
	
	%%=============================================================%%
	%% Prefix	-> \pfx{Dr}
	%% GivenName	-> \fnm{Joergen W.}
	%% Particle	-> \spfx{van der} -> surname prefix
	%% FamilyName	-> \sur{Ploeg}
	%% Suffix	-> \sfx{IV}
	%% NatureName	-> \tanm{Poet Laureate} -> Title after name
	%% Degrees	-> \dgr{MSc, PhD}
	%% \author*[1,2]{\pfx{Dr} \fnm{Joergen W.} \spfx{van der} \sur{Ploeg} \sfx{IV} \tanm{Poet Laureate} 
		%%                 \dgr{MSc, PhD}}\email{iauthor@gmail.com}
	%%=============================================================%%
	
	\author[1]{\fnm{Abderrahman} \sur{Baza}}\email{abderrahmane.baza@gmail.com}
	\equalcont{These authors contributed equally to this work.}
	
	\author*[2]{\fnm{Mohamed} \sur{Rossafi}}\email{rossafimohamed@gmail.com; mohamed.rossafi1@uit.ac.ma}
	\equalcont{These authors contributed equally to this work.}
	
	\author[3]{\fnm{Mohammed} \sur{Mouniane}}\email{mohammed.mouniane@uit.ac.ma}
	\equalcont{These authors contributed equally to this work.}

	\affil[1]{\orgdiv{Laboratory Analysis, Geometry and Applications, Department of Mathematics}, \orgname{Faculty of Sciences, University Ibn Tofail}, \orgaddress{\city{Kenitra}, \country{Morocco}}}
	
	\affil[2]{\orgdiv{Laboratory Partial Differential Equations, Spectral Algebra and Geometry}, \orgname{Higher School of Education and Training, Ibn Tofail University}, \orgaddress{\city{Kenitra}, \country{Morocco}}}
	
	\affil[3]{\orgdiv{Laboratory Analysis, Geometry and Applications, Department of Mathematics}, \orgname{Faculty of Sciences, University Ibn Tofail}, \orgaddress{\city{Kenitra}, \country{Morocco}}}

	\abstract{
		In this paper, we investigate some hyperstability results, inspired by the concept of Ulam stability, for the following functional equations:
		\begin{equation}
			\varphi(x+y)+\varphi(x-y)=2\varphi(x)+2\varphi(y)
		\end{equation}
		\begin{equation}
			\varphi(ax+by)=A\varphi(x)+B\varphi(y)+C
		\end{equation}
		\begin{equation}\label{eqnd}
			f\left(\sum_{i=1}^{m}x_{i}\right)+\sum_{1\leq i<j\leq m}f\big(x_{i}-x_{j}\big)=m\sum_{i=1}^{m}f(x_{i})
		\end{equation}
		in modular spaces.}


	
	\maketitle

\section{Introduction and preliminaries}

A fundamental question in the theory of functional equations is: Under what conditions is a function that satisfies approximately a functional equation necessarily close to an exact solution of that equation? If such conditions are met, the equation is said to be stable.
The initial stability problem related to group homomorphisms was proposed by Ulam in 1940 (see \cite{ulam}). In 1941, the investigation into the stability theory of functional equations began with Hyers' pioneering work. In his paper \cite{hyers}, Hyers provided a partial affirmative response to Ulam's question concerning the additive functional equation, specifically when the groups are Banach spaces. This result was subsequently generalized in 1950 by Aoki \cite{aoki}, Bourgin \cite{bourgin}, and Rassias \cite{ras}, who extended the findings to include additive and linear mappings by addressing the unbounded Cauchy difference. Further advancing this field, G\u{a}vru\c{t}\u{a} \cite{gav} in 1994 introduced a broader generalization of the results by Rassias et al., replacing the bounded condition with a general control positive function $\varphi(x, y)$ to establish the existence of a unique linear mapping.\\
During this period, a special form of stability known as hyperstability was identified. Hyperstability of a functional equation requires that any mapping approximately satisfying the equation (in some defined sense) must indeed be an exact solution. Although the term "hyperstability" was first introduced by Maksa in 2001 \cite{maksa}, its earliest occurrence dates back to 1949, as noted by Bourgin in \cite{bourgin1}. Among the most renowned methods for proving the hyperstability of functional equations is the fixed point approach, which has gained prominence over the past two decades thanks to the contributions of J. Brzd\c{e}k and K. Ciepli\'{n}ski \cite{H1} to the fixed point theory, followed by various other authors. For further reference, see \cite{H2, H3, H4, H5, H6, H7, H8}. 
Additionally, numerous papers have been published addressing the hyperstability of functional equations (see, e.g., \cite{H9, H10, H11, H12, H13}).
The functional equation
\begin{equation}\label{eqquad}
	f(x+y)+f(x-y)=2f(x)+2f(y)
\end{equation}
is called a quadratic functional equation, and 
any solution of this equation is referred to as a quadratic mapping. 
The problem of Hyers-Ulam stability for the quadratic functional equation was established by Skof \cite{skof} for mappings $f: X \to Y $, where $X$ is a normed space and $Y$ is a Banach space. Cholewa \cite{chol} observed that Skof's theorem remains valid if the domain $X$  is replaced by an Abelian group. 
Additionally, Czerwik \cite{cz} demonstrated the Hyers-Ulam-Rassias stability of the quadratic functional equation.\\
Research on modular and modular spaces as extensions of normed spaces was initiated by Nakano \cite{ref-26}. Since the 1950s, numerous eminent mathematicians \cite{ref-27,ref-30,ref-31} have made significant contributions to this field. Orlicz spaces and interpolation theory are two examples of applications of modular spaces in \cite{ref-32,ref-33,ref-30}.
We now present the definition, properties and usual terminologies of the theory of modular spaces.
\begin{definition}\label{Definition1.1}
	Let $Y$ be an arbitrary vector space. A functional $\rho: Y \rightarrow[0, \infty)$ is called a modular, if for arbitrary $u, v \in Y$, the following conditions hold:
	\begin{enumerate}
		\item $\rho(u)=0$ if and only if $u=0$.
		\item $\rho(\alpha u)=\rho(u)$ for every scalar $\alpha$ with $|\alpha|=1$.
		\item $\rho(\alpha u+\beta v) \leq \rho(u)+\rho(v)$ if and only if $\alpha+\beta=1$ and $\alpha ,\beta \geq 0$.\label{item1}\\
		If condition \eqref{item1} is replaced by:
		\item $\rho(\alpha u+\beta v) \leq \alpha \rho(u)+\beta \rho(v)$ if and only if $\alpha+\beta=1$ and $\alpha, \beta \geq 0$, then we say that $\rho$ is a convex modular.
	\end{enumerate}
	A modular $\rho$ defines a corresponding modular space, denoted by $Y_\rho$,  which is given by:
	\begin{equation*}
		Y_\rho=\{u \in Y: \rho(\lambda u) \rightarrow 0 \text { as } \lambda \rightarrow 0\}.
	\end{equation*}
	A function modular is said to satisfy the $\Delta_s$-condition if there exists $\tau_{s} > 0$ such that\\ $\rho(s u) \leq \tau_{s} \rho(u)$ for all $u \in Y_\rho$.
\end{definition}
\begin{definition}
	Let $\left\{u_n\right\}$ be a sequence in $Y_\rho$ and let $u \in Y_\rho$. Then
	\begin{enumerate}
		\item 
		The sequence $\left\{u_n\right\}$ is said to be $\rho$-convergent to $u$, denoted as $u_n \rightarrow u$, if $\rho\left(u_n-u\right) \rightarrow 0$ as $n \rightarrow \infty$.
		\item
		The sequence $\left\{u_n\right\}$ is called $\rho$-Cauchy if $\rho\left(u_n-u_m\right) \rightarrow 0$ as $n, m \rightarrow \infty$.
		\item
		$Y_\rho$ is said to be $\rho$-complete if every $\rho$-Cauchy sequence in $Y_\rho$ is $\rho$-convergent.
	\end{enumerate}
\end{definition}
\begin{proposition}
	In a modular space,
	\begin{itemize}
		\item If $u_n \overset{\rho}{\to} u$ and a is a constant vector, then $u_n+a \overset{\rho}{\to} u+a$.
		\item If $u_n \overset{\rho}{\to} u$ and $v_n \overset{\rho}{\to} v$ then $\alpha u_n + \beta v_n \overset{\rho}{\to} \alpha u+ \beta v$, where $\alpha+\beta \leq 1$ and $\alpha,\beta \geq 0$. 
	\end{itemize}
\end{proposition}

\begin{remark}
	Note that $\rho(u)$ is an increasing function, for all $u \in X$. 
	Suppose $0<a<b$. Then,  by property (4) of Definition \ref{Definition1.1} with $v=0$, we have $\rho(a x)=\rho\left(\dfrac{a}{b} b u\right) \leq \rho(b u)$ for all $u \in Y$. 
	Moreover, if $\rho$ is a convex modular on $Y$ and $|\alpha| \leq 1$, then $\rho(\alpha u) \leq \alpha \rho(u)$.

	In general, if $\lambda_i \geq 0$ and $\sum_{i=1}^{n}\lambda_i=1$, 	
	then
	$$\rho (\lambda_1 u_1+\lambda_2 u_2+\dots+\lambda_n u_n) \leq \lambda_1 \rho(u_1)+\lambda_2 \rho(u_2)+\dots+\lambda_n \rho(u_n).$$
	If $\{u_n\}$ is $\rho$-convergent to $u$, then $\{ c u_n \}$ is $\rho$-convergent to $cu$, where $|c| \leq 1$.
	However, the $\rho$-convergence of a sequence $\{u_n\}$ to $u$ does not imply that $\{\alpha u_n\}$ is $\rho$-convergent to $\alpha u$ for scalars $\alpha $ with $|\alpha|>1$.
	
	If $\rho$ is a convex modular satisfying the $\Delta_s$-condition with $0 <\tau_{s}<s$, then
	$$\rho(u) \leq \tau_{s} \rho(\dfrac{1}{s} u) \leq \dfrac{\tau_{s}}{s} \rho(u)\text{ for all }u.$$
	Hence, $\rho=0$. Consequently, we must have $\tau_{s} \geq s$ if $\rho$ is convex modular. 
\end{remark}

\section{Hyperstability of Equation \eqref{eqquad}}
In this section, we investigate the hyperstability of Eq. \eqref{eqquad} in modular spaces. Specifically, we will prove that every approximate solution of Eq. \eqref{eqquad}, under some conditions, is an exact solution to it. Throughout this section, we assume that $X$ is a vector space over a field $\mathbb{K}$, and $Y_{\rho}$ is a convex modular space satisfying the $\Delta_{2}$-condition.
\begin{theorem}\label{th2}
	Suppose that $E$ is a non-empty subset of $X$ that is symmetric with respect to $0$ and satisfies the conditions $x + y, x - y \in E$ and $kx \in E$ for all $x, y \in E$ and all $k \in \mathbb{K}$. Let $\alpha: E^{2} \to [0, \infty)$ be a function such that
	\begin{equation}\label{eq1t2}
		\lim_{n\rightarrow\infty}\alpha(x,nx)=0,
	\end{equation}
	and
	\begin{equation*}%\label{eq2t2}
		\lim_{n\rightarrow\infty}\alpha(nx,ny)=0,
	\end{equation*}
	for all $x,y\in E$. Let $\varphi:E\to Y_{\rho}$ be a mapping satisfying
	\begin{equation*}\label{eq4t2}
		\rho(\varphi(x+y)+\varphi(x-y)-2\varphi(x)-2\varphi(y)\Big)\leq \alpha (x,y),
	\end{equation*}
	for all $x,y\in E$. Then $\varphi$ is quadratic on $E$.
\end{theorem}

\begin{proof}
	Letting $y=nx$ in \eqref{eq1t2}, we obtain
	\begin{equation*}\label{eq5t2}
		\rho\left(\frac{1}{2}\varphi(x+nx)+\frac{1}{2}\varphi(x-nx)-\varphi(x)-\varphi(nx)\right)\leq \frac{1}{2}\alpha (x,nx),
	\end{equation*}
	for all $x\in E$. 
	Hence, for all $x, y \in E$, we have
	\begin{equation*}
		\rho\left(\frac{1}{2}\varphi(y+ny)+\frac{1}{2}\varphi(y-ny)-\varphi(y)-\varphi(ny)\right)\leq \frac{1}{2}\alpha (y,ny),
	\end{equation*}
	\begin{multline*}
		\rho\left(\frac{1}{2}\varphi\big(x+y+n(x+y)\big)+\frac{1}{2}\varphi\big(x+y-n(x+y)\big)-\varphi(x+y)-\varphi\big(n(x+y)\big)\right)\\ \leq\frac{1}{2} \alpha (x+y,n(x+y)),
	\end{multline*}
	and
	\begin{multline*}
		\rho\left(\frac{1}{2}\varphi\big(x-y+n(x-y)\big)+\frac{1}{2}\varphi\big(x-y-n(x-y)\big)-\varphi(x-y)-\varphi\big(n(x-y)\big)\right)\\\leq \frac{1}{2}\alpha (x-y,n(x-y)),
	\end{multline*}
	for all $x,y\in E$. Letting $n\rightarrow\infty$, we get
	\begin{equation*}
		\varphi(x)=\rho-\lim_{n\rightarrow\infty}\left(\frac{1}{2}\varphi\big((n+1)x\big)+\frac{1}{2}\varphi\big((1-n)x\big)-\varphi(nx)\right),
	\end{equation*}
	\begin{equation*}
		\varphi(y)=\rho-\lim_{n\rightarrow\infty}\left(\frac{1}{2}\varphi\big((n+1)y\big)+\frac{1}{2}\varphi\big((1-n)y\big)-\varphi(ny)\right),
	\end{equation*}
	
	\begin{equation*}
		\varphi(x+y)=\rho-\lim_{n\rightarrow\infty}\left(\frac{1}{2}\varphi\big((n+1)(x+y)\big)+\frac{1}{2}\varphi\big((1-n)(x+y)\big)-\varphi\big(n(x+y)\big)\right),
	\end{equation*}
	\begin{equation*}
		\varphi(x-y)=\rho-\lim_{n\rightarrow\infty}\left(\frac{1}{2}\varphi\big((n+1)(x-y)\big)+\frac{1}{2}\varphi\big((1-n)(x-y)\big)-\varphi\big(n(x-y)\big)\right),
	\end{equation*}
	Then, we have
	\begin{align*}
			&\rho\left(\frac{1}{7}\varphi(x+y) + \frac{1}{7}\varphi(x-y) - \frac{2}{7}\varphi(x) - \frac{2}{7}\varphi(y)\right)
			\\ 
			&\leq \frac{1}{7} \rho\left(\varphi(x+y) - \left(\frac{1}{2}\varphi((n+1)(x+y)) + \frac{1}{2}\varphi((1-n)(x+y)) - \varphi(n(x+y))\right)\right)\\
		&+ \frac{1}{7} \rho\left(\varphi(x-y) - \left(\frac{1}{2}\varphi((n+1)(x-y)) + \frac{1}{2}\varphi((1-n)(x-y)) - \varphi(n(x-y))\right)\right)\\
		&+ \frac{2}{7} \rho\left(\varphi(x) - \left(\frac{1}{2}\varphi((n+1)x) + \frac{1}{2}\varphi((1-n)x) - \varphi(nx)\right)\right)\\
		&+ \frac{2}{7} \rho\left(\varphi(y) - \left(\frac{1}{2}\varphi((n+1)y) + \frac{1}{2}\varphi((1-n)y) - \varphi(ny)\right)\right)\\
		&+ \frac{1}{7} \rho\left(\frac{1}{2} \varphi((n+1)(x+y)) + \frac{1}{2} \varphi((1-n)(x+y)) - \varphi(n(x+y)) \right)\\
		&+ \frac{1}{7} \rho\left(\frac{1}{2} \varphi((n+1)(x-y)) + \frac{1}{2} \varphi((1-n)(x-y)) - \varphi(n(x-y)) \right).\\
		&\leq \frac{1}{7} \rho\left(\varphi(x+y) - \left(\frac{1}{2} \varphi((n+1)(x+y)) + \frac{1}{2} \varphi((1-n)(x+y)) - \varphi(n(x+y))\right)\right)\\
		&+ \frac{1}{7} \rho\left(\varphi(x-y) - \left(\frac{1}{2} \varphi((n+1)(x-y)) + \frac{1}{2} \varphi((1-n)(x-y)) - \varphi(n(x-y))\right)\right)\\
		&+ \frac{2}{7} \rho\left(\varphi(x) - \left(\frac{1}{2} \varphi((n+1)x) + \frac{1}{2} \varphi((1-n)x) - \varphi(nx)\right)\right)\\
		&+ \frac{2}{7} \rho\left(\varphi(y) - \left(\frac{1}{2} \varphi((n+1)y) + \frac{1}{2} \varphi((1-n)y) - \varphi(ny)\right)\right)\\
		&+ \frac{\tau}{14} \rho\left(\frac{1}{2} \varphi((n+1)(x+y)) + \frac{1}{2} \varphi((n+1)(x-y)) - \varphi((n+1)x) - \varphi((n+1)y)\right)\\
		&+ \frac{\tau}{14} \rho\left(\frac{1}{2} \varphi((1-n)(x+y)) + \frac{1}{2} \varphi((1-n)(x-y)) - \varphi((1-n)x) - \varphi((1-n)y)\right).
	\end{align*}
	
	\begin{align*}
		&\leq \frac{1}{7} \rho\left(\varphi(x+y) - \left(\frac{1}{2} \varphi((n+1)(x+y)) + \frac{1}{2} \varphi((1-n)(x+y)) - \varphi(n(x+y))\right)\right)\\
		&+ \frac{1}{7} \rho\left(\varphi(x-y) - \left(\frac{1}{2} \varphi((n+1)(x-y)) + \frac{1}{2} \varphi((1-n)(x-y)) - \varphi(n(x-y))\right)\right)\\
		&+ \frac{2}{7} \rho\left(\varphi(x) - \left(\frac{1}{2} \varphi((n+1)x) + \frac{1}{2} \varphi((1-n)x) - \varphi(nx)\right)\right)\\
		&+ \frac{2}{7} \rho\left(\varphi(y) - \left(\frac{1}{2} \varphi((n+1)y) + \frac{1}{2} \varphi((1-n)y) - \varphi(ny)\right)\right)\\
		&+ \frac{\tau}{28} \alpha((n+1)x, (n+1)y)\\
		&+ \frac{\tau^2}{28} \rho\left(\frac{1}{2} \varphi((1-n)(x+y)) + \frac{1}{2} \varphi((1-n)(x-y)) - \varphi((1-n)x) - \varphi((1-n)y)\right)\\
		&+ \frac{\tau^2}{28} \rho\left(\frac{1}{2} \varphi(n(x+y)) + \frac{1}{2} \varphi(n(x-y)) - \varphi(nx) - \varphi(ny)\right).
	\end{align*}
		for all $x,y\in E$.  Letting $n\rightarrow\infty$, we get
	\begin{equation*}
		\rho\left(\frac{1}{7}\varphi(x+y)+\frac{1}{7}\varphi(x-y)-\frac{2}{7}\varphi(x)-\frac{2}{7}\varphi(y)\right)\rightarrow 0
	\end{equation*}
	which implies that
	\begin{equation*}
		\varphi(x+y)+\varphi(x-y)=2\varphi(x)+2\varphi(y),
	\end{equation*}
	for $x,y\in E$, which means that $\varphi$ is quadratic on $E$.
\end{proof}
\begin{corollary}
	Let $\theta\geq0$ and $p$ and $q$ be two real numbers such that  $p+q<0$. Let $\varphi:X\to Y_{\rho}$ be a mapping satisfying
	\begin{equation*}\label{eqc1}
		\rho\Big(\varphi(x+y)+\varphi(x-y)-2\varphi(x)-2\varphi(y)\Big)\leq \theta\|x\|^{p}\|y\|^{q},
	\end{equation*}
	for all $x,y\in E$. Then $\varphi$ is quadratic on $E$.
\end{corollary}
\begin{proof}
	Let $\alpha(x,y)=\theta\|x\|^{p}\|y\|^{q}$ for all $x,y\in E$. Since $p+q<0$, at least one of $p$ or $q$ must be negative. Assume that $q<0$, we have
	\begin{equation*}
		\lim_{n\rightarrow\infty}\alpha(x,ny)=\lim_{n\rightarrow\infty}\theta n^{q}\|x\|^{p}\|y\|^{q}=0
	\end{equation*}
	and
	\begin{equation*}
		\lim_{n\rightarrow\infty}\alpha(nx,ny)=\lim_{n\rightarrow\infty}\theta n^{p+q}\|x\|^{p}\|y\|^{q}=0.
	\end{equation*}
	Therefore, the conditions in Theorem \ref{th2} hold which means that $\varphi$ is quadratic on $E$.
\end{proof}

\section{Hyperstability of a general linear functional equation}
The following section investigates the hyperstability of a general linear functional equation, providing conditions under which approximate solutions are guaranteed to coincide exactly with the functional equation in a modular space setting.
\begin{theorem}\label{th3}
	Suppose that  $E$  is a non-empty subset of $ X$, symmetric with respect to 0 , and satisfying $x+y, x-y \in E$ and $k x \in E$ for all $x, y \in E$ and all $k \in \mathbb{K}$.
	Let $a,b\in \mathbb{K}\setminus\{0\}$ and $\alpha:E^{2}\to [0,\infty)$ be a function such that:
	\begin{equation*}\label{eq1th3}
		\lim_{n\rightarrow\infty}\alpha\big(a^{-1}(n+1)x,-b^{-1}nx\big)=0
	\end{equation*}
	and
	\begin{equation*}\label{eq2th3}
		\lim_{n\rightarrow\infty}\alpha\big(nx,ny\big)=0,
	\end{equation*}
	for all $x,y\in E\setminus\{0\}$. Let $A,B\in \mathbb{R}$, $C\in Y_{\rho}$ such that $|A|\leq 1$ and $|B|\leq 1$ and let $\varphi:X\to Y_{\rho}$ satisfying:
	\begin{equation}\label{eq3th2}
		\rho\big(\varphi(ax+by)-A\varphi(x)-B\varphi(y)-C\big)\leq \alpha(x,y),
	\end{equation}
	for all $x,y\in M_{\alpha}=\{z\in E:\|z\|\geq \alpha\}$ for some $\alpha>0$. Then $\varphi$ satisfies
	\begin{equation*}\label{eq4th3}
		\varphi(ax+by)=A\varphi(x)+B\varphi(y)+C
	\end{equation*}
	and
	\begin{equation*}\label{eq5th3}
		(A+B)\varphi(0)=A\varphi(x)+B\varphi\big(-ab^{-1}x\big),
	\end{equation*}
	for all $x,y\in E$.
\end{theorem}
\begin{proof}
	Substituting $x$ by $a^{-1}(n+1)x$ and $y$ by $-b^{-1}nx$ in \eqref{eq3th2}, we get
	\begin{equation}\label{eq6th2}
		\rho\big(\varphi(x)-A\varphi(a^{-1}(n+1)x)-B\varphi(-b^{-1}nx)-C\big)\leq \alpha\big(a^{-1}(n+1)x,-b^{-1}nx\big),
	\end{equation}
	for all $x\in E\setminus\{0\}$ and all positive integers $n\geq m$, where $a^{-1}(m+1)x,\;-b^{-1}mx\in M_{\alpha}$. Letting  $n\rightarrow\infty$ in \eqref{eq6th2}, we obtain
	\begin{equation*}
		\varphi(x)=\rho-\lim_{n\rightarrow\infty}\Big[A\varphi\big(a^{-1}(n+1)x\big)+B\varphi\big(-b^{-1}nx\big)+C\Big],\;\;\;x\in E\setminus\{0\}.
	\end{equation*}
	Therefore, we have
	\begin{align*}
		&\rho\Bigg(\frac{1}{4}\varphi(ax+by)-\frac{A}{4}\varphi(x)-\frac{B}{4}\varphi(y)-\frac{C}{4}\Bigg)\\
		&\leq \frac{1}{4}\rho\Big(\varphi(ax+by)-\big(A\varphi(a^{-1}(n+1)(ax+by))+B\varphi(-b^{-1}n(ax+by))+C\big)\Big)\\
		&+\frac{1}{4}\rho\Big(A\varphi(x)-\big(A^{2}\varphi(a^{-1}(n+1)x)+AB\varphi(-b^{-1}nx)+AC\big)\Big)\\
		&+\frac{1}{4}\rho\Big(B\varphi(y)-\big(AB\varphi(a^{-1}(n+1)y)+B^{2}\varphi(-b^{-1}nx)+BC\big)\Big)\\
		&+\frac{1}{4}\rho\Big(A\varphi(a^{-1}(n+1)(ax+by))+B\varphi(-b^{-1}n(ax+by))+C-A^{2}\varphi(a^{-1}(n+1)x)\\
		&\;\;\;-AB\varphi(-b^{-1}nx)-AC-AB\varphi(a^{-1}(n+1)y)-B^{2}\varphi(-b^{-1}ny)-BC-C\Big)\\
		&\leq \frac{1}{4}\rho\Big(\varphi(ax+by)-\big(A\varphi(a^{-1}(n+1)(ax+by))+B\varphi(-b^{-1}n(ax+by))+C\big)\Big)\\
		&+\frac{1}{4}\rho\Big(A\varphi(x)-\big(A^{2}\varphi(a^{-1}(n+1)x)+AB\varphi(-b^{-1}nx)+AC\big)\Big)\\
		&+\frac{1}{4}\rho\Big(B\varphi(y)-\big(AB\varphi(a^{-1}(n+1)y)+B^{2}\varphi(-b^{-1}nx)+BC\big)\Big)\\
		&+\frac{|A|\tau}{8}\rho\Big(\varphi(a^{-1}(n+1)(ax+by))-A\varphi(a^{-1}(n+1)x)-B\varphi(a^{-1}(n+1)y)-C\Big)\\
		&+\frac{|A|\tau}{8}\rho\Big(\varphi(-b^{-1}n(ax+by))-A\varphi(-b^{-1}nx)-B\varphi(-b^{-1}ny)-C\Big)\\
		&\leq \frac{1}{4}\rho\Big(\varphi(ax+by)-\big(A\varphi(a^{-1}(n+1)(ax+by))+B\varphi(-b^{-1}n(ax+by))+C\big)\Big)\\
		&+\frac{1}{4}\rho\Big(A\varphi(x)-\big(A^{2}\varphi(a^{-1}(n+1)x)+AB\varphi(-b^{-1}nx)+AC\big)\Big)\\
		&+\frac{1}{4}\rho\Big(B\varphi(y)-\big(AB\varphi(a^{-1}(n+1)y)+B^{2}\varphi(-b^{-1}nx)+BC\big)\Big)\\
		&+\frac{|A|\tau}{8}\alpha\big(a^{-1}(n+1)x,a^{-1}(n+1)y\big)+\frac{|A|\tau}{8}\alpha\big(-b^{-1}nx,-b^{-1}ny\big)\\
		&\rightarrow 0\;\;\text{as}\;\;\;n\rightarrow\infty
	\end{align*}
	Hence, we obtain
	\begin{equation*}
		\varphi(ax+by)=A\varphi(x)+B\varphi(y)+C.
	\end{equation*}
	Moreover, we have
	\begin{align*}
		&\rho\left(\frac{1}{3}(A+B)\varphi(0) - \frac{1}{3}A\varphi(x) - \frac{1}{3}B\varphi\left(-ab^{-1}x\right)\right) \\
		&\leq \frac{1}{3}\rho\left(A\varphi(x) - \left(A^{2}\varphi\left(a^{-1}(n+1)x\right) + AB\varphi\left(-b^{-1}nx\right) + AC\right)\right) \\
		&\quad + \frac{1}{3}\rho\left(B\varphi\left(-ab^{-1}nx\right) - \left(AB\varphi\left(-b^{-1}(n+1)x\right) + B^{2}\varphi\left(ab^{-2}nx\right) + BC\right)\right) \\
		&\quad + \frac{1}{3}\rho\left((A+B)\varphi(0) - A^{2}\varphi\left(a^{-1}(n+1)x\right) + AB\varphi\left(-b^{-1}nx\right) - AC \right. \\
		&\quad \left. - AB\varphi\left(-b^{-1}(n+1)x\right) - B^{2}\varphi\left(ab^{-2}nx\right) - BC\right) \\
		&\leq \frac{1}{3}\rho\left(A\varphi(x) - \left(A^{2}\varphi\left(a^{-1}(n+1)x\right) + AB\varphi\left(-b^{-1}nx\right) + AC\right)\right) \\
		&\quad + \frac{1}{3}\rho\left(B\varphi\left(-ab^{-1}nx\right) - \left(AB\varphi\left(-b^{-1}(n+1)x\right) + B^{2}\varphi\left(ab^{-2}nx\right) + BC\right)\right) \\
		&\quad + \frac{|A|\tau}{6}\rho\left(\varphi(0) - A\varphi\left(a^{-1}(n+1)x\right) - B\varphi\left(-b^{-1}(n+1)x\right) - C\right) \\
		&\quad + \frac{|B|\tau}{6}\rho\left(\varphi(0) - A\varphi\left(-b^{-1}nx\right) - B\varphi\left(ab^{-2}nx\right) - C\right) \\
		&\leq \frac{1}{3}\rho\left(A\varphi(x) - \left(A^{2}\varphi\left(a^{-1}(n+1)x\right) + AB\varphi\left(-b^{-1}nx\right) + AC\right)\right) \\
		&\quad + \frac{1}{3}\rho\left(B\varphi\left(-ab^{-1}nx\right) - \left(AB\varphi\left(-b^{-1}(n+1)x\right) + B^{2}\varphi\left(ab^{-2}nx\right) + BC\right)\right) \\
		&\quad + \frac{|A|\tau}{6}\alpha\left(a^{-1}(n+1)x, -b^{-1}(n+1)x\right) + \frac{|B|\tau}{6}\alpha\left(-b^{-1}nx, ab^{-2}nx\right) \\
		&\rightarrow 0 \quad \text{as} \quad n \to \infty.
	\end{align*}
		Hence,
	\begin{equation}\label{eqA+B}
		(A+B)\varphi(0)=A\varphi(x)+B\varphi(-ab^{-1}x\big).
	\end{equation}
	Now, if we replace $x$ by $bnx$ and $y$ by $-anx$ in the inequality \eqref{eq3th2}, we obtain:
	$$
	\rho\big(\varphi(0)-A\varphi(bnx)-B\varphi(-anx)-C\big)\leq \alpha(bnx,-anx) \longrightarrow 0 \quad \text{as} \quad n \to \infty.
	$$
	Hence, 
	$$
	\varphi(0) = \rho-\lim_{n\to\infty} \big[A\varphi(bnx) + B\varphi(-anx) + C\big].
	$$
	On the other hand, if we replace $x$ by $bnx$ in \eqref{eqA+B}, we get:
	$$(A+B)\varphi(0)=A\varphi(bnx)+B\varphi(-anx).$$
	Then $$\rho\big((1-A-B)\varphi(0)-C\big)=\rho\big(\varphi(0)-A\varphi(bnx)-B\varphi(-anx)-C)$$
	and therefore,
	$$C=(1-A-B)\varphi(0).$$
\end{proof}
\begin{corollary}
	Let $a,b\in \mathbb{K}\setminus\{0\}$ and let $\varphi:E\to Y_{\rho}$ be a function. Take $\theta$,$\theta^{'}\geq0$, and $p,q,r$ be real numbers. Suppose that one of the following conditions holds:\\
	$(i)p+q+r\leq0$ and 
	\begin{equation*}\label{eq3th3}
		\rho\big(\varphi(ax+by)-A\varphi(x)-B\varphi(y)-C\big)\leq\|x\|^{p}\|y\|^{q}\big(\theta\|x+y\|^{r}+\theta\|x-y\|^{r}\big)
	\end{equation*}
	$(ii)p+q\leq0$ and 
	\begin{equation*}\label{eq3th3}
		\rho\big(\varphi(ax+by)-A\varphi(x)-B\varphi(y)-C\big)\leq\theta\|x\|^{p}\|y\|^{q}
	\end{equation*}
	$(iii)p,q<0$ and 
	\begin{equation*}\label{eq3th3}
		\rho\big(\varphi(ax+by)-A\varphi(x)-B\varphi(y)-C\big)\leq\theta\|x\|^{p}+\theta^{'}\|y\|^{q}
	\end{equation*}
	Then $\varphi$ satisfies
	\begin{equation*}\label{eq4th3}
		\varphi(ax+by)=A\varphi(x)+B\varphi(y)+C
	\end{equation*}
	and
	\begin{equation*}\label{eq5th3}
		(A+B)\varphi(0)=A\varphi(x)+B\varphi\big(-ab^{-1}x\big),
	\end{equation*}
	for all $x,y\in M_{\alpha}=\{z\in E:\|z\|\geq \alpha\}$ for some $\alpha>0$.
\end{corollary}
\begin{remark}
	\begin{itemize}
		\item [(i)] If $a=b=A=B=1$ and $C=0$, we obtain the hyperstability result for the additive functional equation
		$$
		\varphi(x+y)=\varphi(x)+\varphi(y)
		$$
		in modular space.\\
		\item [(ii)] if $a=b=A=B=\dfrac{1}{2}$ and $C=0$, we obtain the hyperstability result for the Jensen functional equation
		$$
		\varphi\big(\frac{x+y}{2}\big)=\frac{1}{2}\varphi(x)+\frac{1}{2}\varphi(y)
		$$
		in modular space.
	\end{itemize}
\end{remark}
\section{Hyperstability of the $n$-dimensional quadratic functional equation}
In this section, we study the hyperstability of the $n$-dimensional quadratic functional equation.
\begin{theorem}\label{th1}
	Suppose that  $E$   is a non-empty subset of $X$ that is symmetric with respect to $0$ and satisfies $x+y, x-y \in E$ and $k x \in E$ for all $x, y \in E$ and all $k \in \mathbb{K}$. Let $f$ : $E \rightarrow Y$ and $\varphi: E^m \rightarrow[0, \infty)$ be two functions that satisfy the following conditions
	\begin{equation}\label{cond1th3}
		\lim_{n\rightarrow\infty}\varphi\big(x,nx,\cdots,nx\big)=0,
	\end{equation}
	\begin{equation*}\label{cond2th3}
		\lim_{n\rightarrow\infty}\varphi\big(nx_{1},nx_{2},\cdots,nx_{m}\big)=0
	\end{equation*}
	and
	\begin{equation}\label{eq11th3}
		\rho\left(f\left(\sum_{i=1}^{m}x_{i}\right)+\sum_{1\leq i<j\leq m}f(x_{i}-x_{j})- m\sum_{i=1}^{m}f(x_{i})\right)\leq \varphi(x_{1},x_{2},\cdots,x_{m}),
	\end{equation}
	for all $x_{1},x_{2},\cdots,x_{m}\in E$. Then $f$ satisfies equation \eqref{eqnd} on $E$.
\end{theorem}
\begin{proof}
	Letting $x_{i}=nx$ with $i\geq 2$ and $n\in \mathbb{N}$ in \eqref{eq11th3}, we obtain 
	\begin{multline*}%\label{eq2th3}
		\rho\left(f\big((1+mn-n)x\big)+(m-1)f((1-n)x)-m(m-1)f(nx)-mf(x)\right)  \\ \leq \varphi\big(x,nx,\cdots,nx\big),
	\end{multline*}
	Hence
	\begin{multline*}%\label{eq2ath3}
		\rho\left(\frac{1}{m}f\big((1+mn-n)x\big)+\frac{m-1}{m}f((1-n)x)-(m-1)f(nx)-f(x)\right) \\ \leq \frac{1}{m}\varphi\big(x,nx,\cdots,nx\big),
	\end{multline*}
	for all $x\in X$ and all $n\in \mathbb{N}$. In view of (\ref{cond1th3}), we deduce that
	\begin{equation*}\label{eq3th3}
		f(x)=\rho-\lim_{n\rightarrow\infty}\bigg[\dfrac{1}{m}f\big((1+mn-n)x\big)+\dfrac{(m-1)}{m}f((1-n)x)-(m-1)f(nx)\bigg],
	\end{equation*}
	for all $x\in X$. Therefore,
	\begin{multline*}\label{eq4th3}
	   f\left(\sum_{i=1}^{m}x_{i}\right)=\rho-\lim_{n\rightarrow\infty}\Bigg\{\frac{1}{m}f\left((1+mn-n)\sum_{i=1}^{m}x_{i}\right)
	   \\
	   -\frac{(m-1)}{m}f\left((1-n)\sum_{i=1}^{m}x_{i}\right)-(m-1)f\left(n\sum_{i=1}^{m}x_{i}\right)\Bigg\},
	\end{multline*}
	\begin{equation*}\label{eq5th3}
		f(x_{i})=\rho-\lim_{n\rightarrow\infty}\Bigg\{f\big((1+mn-n)x_{i}\big)-(m-1)f((1-n)x_{i})-m(m-1)f\big(nx_{i}\big)\Bigg\},
	\end{equation*}
	and
	\begin{align*}\label{eq6th3}
		f\big(x_{i}-x_{j}\big)=\rho-\lim_{n\rightarrow\infty}\Bigg\{&\frac{1}{m}f\big((1+mn-n)(x_{i}-x_{j})\big)+\frac{(m-1)}{m}f\big((1-n)(x_{i}-x_{j})\big)\nonumber\\
		&-(m-1)f\big(n(x_{i}-x_{j})\big)\Bigg\},
	\end{align*}
	for all $x_{1},x_{2},\cdots,x_{m}\in E$. Now, we have:
	\begin{align*}
		& \rho\left[\frac{2}{3 m^{2}-m+4}\left(f\left(\sum_{i=1}^{m} x_{i}\right)+\sum_{1 \leq i<j \leq m} f\left(x_{i}-x_{j}\right)-m \sum_{i=1}^{m} f\left(x_{i}\right)\right)\right. \\
		& \leq \frac{2}{3 m^{2}-m+4} \rho\left[f\left(\sum_{i=1}^{m} x_{i}\right)-\frac{1}{m} f\left((1+m n-n) \sum_{i=1}^{m} x_{i}\right)\right. \\
		& \left.-\frac{m-1}{m} f\left((1-n) \sum_{i=1}^{m} x i\right)+(m-1) f\left(n \sum_{i=1}^{m} x i\right)\right] \\
		& +\frac{2}{3 m^{2}-m+4} \sum_{1 \leq i<j \leq m} \rho\left[f\left(x_{i}-x_{j}\right)-\frac{1}{m} f\left((1+m n-n)\left(x_{i}-x_{j}\right)\right)\right. \\
		& \left.-\frac{m-1}{m} f\left((1-n)\left(x_{i}-x_{j}\right)\right)+(m-1) f\left(n\left(x_{i}-x_{j}\right)\right)\right] \\
		& +\frac{2 m}{3 m^{2}-m+4} \sum_{i=1}^{m} \rho\left[f\left(x_{i}\right)-\frac{1}{m} f\left((1+m n-n) x_{i}\right)\right. \\
		& \left.-\frac{m-1}{m} f\left((1-n) x_{i}\right)+(m-1) f\left(n x_{i}\right)\right] \\
		& +\frac{2}{3 m^{2}-m+4} \rho\left[\frac{1}{m} f\left((1+m n-n) \sum_{i=1}^{m} x_{i}\right)+\frac{m-1}{m} f\left((1-n) \sum_{i=1}^{m} x_{i}\right)\right] \\
		& -(m-1) f\left(n \sum_{i=1}^{m} x_{i}\right)+\frac{1}{m} \sum_{1 \leq i<j \leq m} f\left((1+m n-n)\left(x_{i}-x_{j}\right)\right) \\
		& +\frac{m-1}{m} \sum_{1\leq i<j\leq m} f\left((1-n)\left(x_{i}-x_{j}\right)\right)-(m-1) \sum_{1 \leq i<j \leq m} f\left(n\left(x_{i}-x_{j}\right)\right) \\
		& -\sum_{i=1}^{m} f\left((1+m n-n) x_{i}\right)-(m-1) \sum_{i=1}^{m} f\left((1-n) x_{i}\right) 
		\left.+m(m-1) \sum_{i=1}^{m} f\left(n x_{i}\right)\right]
	\end{align*}
	
	\begin{align*}
		&\leq \frac{2}{3 m^{2}-m+4} \rho\left[f\left(\sum_{i=1}^{m} x_{i}\right)-\frac{1}{m} f\left((1+m n-n) \sum_{i=1}^{m} x_{i}\right)\right. \\
		& \left.-\frac{m-1}{m} f\left((1-n) \sum_{i=1}^{m} x i\right)+(m-1) f\left(n \sum_{i=1}^{m} x i\right)\right] \\
		& +\frac{2}{3 m^{2}-m+4} \sum_{1 \leq i<j \leq m} \rho\left[f\left(x_{i}-x_{j}\right)-\frac{1}{m} f((1+mn-n)(x_{i}-x_{j}))\right. \\
		& \left.-\frac{m-1}{m} f\left((1-n)\left(x_{i}-x_{j}\right)\right)+(m-1) f\left(n\left(x_{i}-x_{j}\right)\right)\right] \\
		& +\frac{2 m}{3 m^{2}-m+4} \sum_{i=1}^{m} \rho\left[f\left(x_{i}\right)-\frac{1}{m} f\left((1+m n-n) x_{i}\right)\right. \\
		& \left.-\frac{m-1}{m} f\left((1-n) x_{i}\right)+(m-1) f\left(n x_{i}\right)\right]\\
		&+\frac{k_{2}}{3 m^{2}-m+4}\rho\bigg\{\frac{1}{m} f\left((1+m n-n) \sum_{i=1}^{m} x_{i}\right) \\
		& +\frac{1}{m} \sum_{1 \leq i<j \leq m} f\left((1+m n-n)\left(x_{i}-x_{j}\right)\right)-\sum_{i=1}^{m} f\left((1+m n-n) x_{i}\right)\bigg\}\\
		&+\frac{k_{2}}{3 m^{2}-m+4}\rho\bigg\{\frac{m-1}{m} f\left((1-n) \sum_{i=1}^{m} x_{i}\right)\\
		&+\frac{m-1}{m} \sum_{1\leq i<j\leq m} f\left((1-n)\left(x_{i}-x_{j}\right)\right)-(m-1) \sum_{i=1}^{m} f\left((1-n) x_{i}\right) \\
		&-(m-1) f\left(n \sum_{i=1}^{m} x_{i}\right)-(m-1) \sum_{1 \leq i<j \leq m} f\left(n\left(x_{i}-x_{j}\right)\right) 
		+m(m-1) \sum_{i=1}^{m} f\left(n x_{i}\right)\bigg\}\\
		&\leq\frac{2}{3 m^{2}-m+4} \rho\left[f\left(\sum_{i=1}^{m} x_{i}\right)-\frac{1}{m} f\left((1+m n-n) \sum_{i=1}^{m} x_{i}\right)\right. \\
		& \left.-\frac{m-1}{m} f\left((1-n) \sum_{i=1}^{m} x_{i}\right)+(m-1) f\left(n \sum_{i=1}^{m} x_{i}\right)\right] \\
		& +\frac{2}{3 m^{2}-m+4} \sum_{1 \leq i<j \leq m} \rho\left[f\left(x_{i}-x_{j}\right)-\frac{1}{m} f((1+mn-n)(x_{i}-x_{j}))\right. \\
		& \left.-\frac{m-1}{m} f\left((1-n)\left(x_{i}-x_{j}\right)\right)+(m-1) f\left(n\left(x_{i}-x_{j}\right)\right)\right] \\
		& +\frac{2 m}{3 m^{2}-m+4} \sum_{i=1}^{m} \rho\left[f\left(x_{i}\right)-\frac{1}{m} f\left((1+m n-n) x_{i}\right)\right. \\
		& \left.-\frac{m-1}{m} f\left((1-n) x_{i}\right)+(m-1) f\left(n x_{i}\right)\right]\\
		&+\frac{k_{2}}{3 m^{3}-m^{2}+4m}\rho\bigg\{ f\left((1+m n-n) \sum_{i=1}^{m} x_{i}\right) \\
		& +\sum_{1 \leq i<j \leq m} f\left((1+m n-n)\left(x_{i}-x_{j}\right)\right)-m\sum_{i=1}^{m} f\left((1+m n-n) x_{i}\right)\bigg\}
	\end{align*}
	
	\begin{align*}
		&+\frac{k_{2}^{2}}{6 m^{2}-2m+8}\rho\bigg\{\frac{m-1}{m} f\left((1-n) \sum_{i=1}^{m} x_{i}\right)\\
		&+\frac{m-1}{m} \sum_{1\leq i<j\leq m} f\left((1-n)\left(x_{i}-x_{j}\right)\right)-(m-1) \sum_{i=1}^{m} f\left((1-n) x_{i}\right)\bigg\} \\
		&+\frac{k_{2}^{2}}{6 m^{2}-2m+8}\rho\bigg\{(m-1) f\left(n \sum_{i=1}^{m} x i\right)+(m-1) \sum_{1 \leq i<j \leq m} f\left(n\left(x_{i}-x_{j}\right)\right)\\
		&- m(m-1) \sum_{i=1}^{m} f\left(n x_{i}\right)\bigg\}\\
		&\leq\frac{2}{3 m^{2}-m+4} \rho\left[f\left(\sum_{i=1}^{m} x_{i}\right)-\frac{1}{m} f\left((1+m n-n) \sum_{i=1}^{m} x_{i}\right)\right. \\
		& \left.-\frac{m-1}{m} f\left((1-n) \sum_{i=1}^{m} x_{i}\right)+(m-1) f\left(n \sum_{i=1}^{m} x_{i}\right)\right] \\
		& +\frac{2}{3 m^{2}-m+4} \sum_{1 \leq i<j \leq m} \rho\left[f\left(x_{i}-x_{j}\right)-\frac{1}{m} f((1+mn-n)(x_{i}-x_{j}))\right. \\
		& \left.-\frac{m-1}{m} f\left((1-n)\left(x_{i}-x_{j}\right)\right)+(m-1) f\left(n\left(x_{i}-x_{j}\right)\right)\right] \\
		& +\frac{2 m}{3 m^{2}-m+4} \sum_{i=1}^{m} \rho\left[f\left(x_{i}\right)-\frac{1}{m} f\left((1+m n-n) x_{i}\right)\right. \\
		& \left.-\frac{m-1}{m} f\left((1-n) x_{i}\right)+(m-1) f\left(n x_{i}\right)\right]\\
		&+\frac{k_{2}}{3 m^{3}-m^{2}+4m}\varphi\left((1+mn-n)x_{1},\cdots,(1+mn-n)x_{n}\right)\\
		&+\frac{k_{2}^{2}(m-1)}{6 m^{3}-2m^{2}+8m}\varphi\left((1-n)x_{1},\cdots,(1-n)x_{n}\right)
		+\frac{k_{2}^{2}k_{m-1}}{6 m^{2}-2m+8}\varphi\left(nx_{1},\cdots,nx_{n}\right)
	\end{align*}
	$$\longrightarrow 0 \text{ as }n \longrightarrow\infty$$
	for all $x_{1},\cdots,x_{m}\in E$. It means that the equation (\ref{eqnd}) is hyperstable on $E$.
\end{proof}

\begin{corollary}
	Let $\theta$ and $p$ be two real numbers such that $\theta\geq0$ and $p<0$. Let $f:E\to Y_{\rho}$ be a mapping satisfying 
	\begin{equation*}\label{eq1cor1}
		\rho\left(f\left(\sum_{i=1}^{m}x_{i}\right)+\sum_{1\leq i<j\leq m}f(x_{i}-x_{j})- m\sum_{i=1}^{m}f(x_{i})\right)\leq \theta \prod_{i=1}^{m}\|x_{i}\|^{p},
	\end{equation*}
	for all $x_{1},\cdots,x_{m}\in E\setminus\{0\}$. Then $f$ satisfies equation \eqref{eqnd} on $E$
\end{corollary}
\begin{proof}
	In Theorem \ref{th1}, we suppose that
	$$
	\varphi(x_{1},\cdots,x_{m}):=\theta\prod_{i=1}^{m}\|x_{i}\|^{p}
	$$
	for all $x_{1},\cdots,x_{m}\in E$. We notice that
	$$
	\lim_{n\rightarrow\infty}\varphi\big(x,nx,\cdots,nx\big)=\lim_{n\rightarrow\infty}\theta n^{(m-1)p}\|x\|^{p}=0 
	$$
	and
	$$
	\lim_{n\rightarrow\infty}\varphi\big(nx,nx,\cdots,nx\big)=\lim_{n\rightarrow\infty}\theta n^{mp}\|x\|^{p}=0,
	$$
	for all $x\in E$.This implies that $f$ satisfies the equation \eqref{eqnd} on $E$.
\end{proof}

\section{Hyperstability of the $n$-dimensional quadratic equation in Banach space}
This last section investigates the hyperstability of the $n$-dimensional quadratic equation within the context of Banach spaces.
\begin{theorem}\label{th1-5}
	Suppose that $E$ is a non-empty subset of $X$, symmetric with respect to $0$, and satisfying $x+y,x-y\in E$ and $kx\in E$ for all $x,y\in E$ and all $k\in\mathbb{K}$. Let
	$\varphi:E^{m}\to [0,\infty)$ be a function such that:
	\begin{equation}\label{cond1th4}
		\lim_{n\rightarrow\infty}\varphi\big(x,nx,\cdots,nx\big)=0
	\end{equation}
	and
	\begin{equation*}\label{cond2th3}
		\lim_{n\rightarrow\infty}\varphi\big(nx_{1},nx_{2},\cdots,nx_{m}\big)=0.
	\end{equation*}
	Let $Y$ be a Banach space, and $f:E\longrightarrow Y$ be a mapping satisfying: 
	\begin{equation}\label{eq1th4}
		\left\|f\left(\sum_{i=1}^{m}x_{i}\right)+\sum_{1\leq i<j\leq m}f(x_{i}-x_{j})- m\sum_{i=1}^{m}f(x_{i})\right\|\leq \varphi(x_{1},x_{2},\cdots,x_{m}),
	\end{equation}
	for all $x_{1},x_{2},\cdots,x_{m}\in E$. Then $f$ satisfies the equation \eqref{eqnd} on $E$ .
\end{theorem}
\begin{proof}
	Let $x_{i}=nx$ with $i\geq 2$ and $n\in \mathbb{N}$ in \eqref{eq1th4}, we get
	\begin{multline*}\label{eq2th3}
		\bigg\|f\big((1+mn-n)x\big)+(m-1)f((1-n)x)-m(m-1)f(nx)-mf(x)\bigg\|  \\  \leq \varphi\big(x,nx,\cdots,nx\big),
	\end{multline*}
	for all $x\in E$, and all $n\in \mathbb{N}$. In view of  \eqref{cond1th4}, we deduce that
	\begin{equation*}\label{eq3th3}
		f(x)=\lim_{n\rightarrow\infty}\dfrac{1}{m}f\big((1+mn-n)x\big)+\dfrac{(m-1)}{m}f((1-n)x)-(m-1)f(nx),
	\end{equation*}
	for all $x\in X$. In other hand, we have 
	\begin{multline*}%\label{eq4th3}
		f\left(\sum_{i=1}^{m}x_{i}\right)=\lim_{n\rightarrow\infty}\Bigg\{\frac{1}{m}f\left((1+mn-n)\sum_{i=1}^{m}x_{i}\right)\\
		-\frac{(m-1)}{m}f\left((1-n)\sum_{i=1}^{m}x_{i}\right)-(m-1)f\left(n\sum_{i=1}^{m}x_{i}\right)\Bigg\},
	\end{multline*}
	\begin{multline*}%\label{eq5th3}
		m\sum_{i=1}^{m}f(x_{i})=\lim_{n\rightarrow\infty}\Bigg\{\sum_{i=1}^{m}f\big((1+mn-n)x_{i}\big)\\-(m-1)\sum_{i=1}^{m}f((1-n)x_{i})-m(m-1)\sum_{i=1}^{m}f\big(nx_{i}\big)\Bigg\},
	\end{multline*}
	and
	\begin{multline*}\label{eq6th3}
		\sum_{1\leq i<j\leq m}f\big(x_{i}-x_{j}\big)=\lim_{n\rightarrow\infty}\Bigg\{\frac{1}{m}\sum_{1\leq i<j\leq m}f\big((1+mn-n)(x_{i}-x_{j})\big)+\\
		\frac{(m-1)}{m}\sum_{1\leq i<j\leq m}f\big((1-n)(x_{i}-x_{j})\big) 
		-(m-1)\sum_{1\leq i<j\leq m}f\big(n(x_{i}-x_{j})\big)\Bigg\},
	\end{multline*}
	for all $x_{1},x_{2},\cdots,x_{m}\in E$. 
	Hence
	\begin{multline*}
		\bigg\|f\left(\sum_{i=1}^{m}x_{i}\right)+\sum_{1\leq i<j\leq m}f\big(x_{i}-x_{j}\big)-m\sum_{i=1}^{m}f(x_{i})\bigg\|
		\\
		=\lim_{n\longrightarrow \infty}\Bigg\|\frac{1}{m}f\left((1+mn-n)\sum_{i=1}^{m}x_{i}\right)
		-\frac{(m-1)}{m}f\left((1-n)\sum_{i=1}^{m}x_{i}\right)\\-(m-1)f\left(n\sum_{i=1}^{m}x_{i}\right)
		+\frac{1}{m}\sum_{1\leq i<j\leq m}f\big((1+mn-n)(x_{i}-x_{j})\big)\\
		+\frac{(m-1)}{m}\sum_{1\leq i<j\leq m}f\big((1-n)(x_{i}-x_{j})\big)
		-(m-1)\sum_{1\leq i<j\leq m}f\big(n(x_{i}-x_{j})\big)\\
		-\sum_{i=1}^{m}f\big((1+mn-n)x_{i}\big)
		+(m-1)\sum_{i=1}^{m}f((1-n)x_{i})+m(m-1)\sum_{i=1}^{m}f\big(nx_{i}\big)\Bigg\|\\
		%%%%%%%%%%%%%%%%%%%%%%%%%%%%%%%%%%%%%%
		\leq \lim_{n\longrightarrow\infty}\sup\left\{\frac{1}{m}\bigg\|f\left((1+mn-n)\sum_{i=1}^{m}x_{i}\right) \right.
		\left.
		 +\sum_{1\leq i<j\leq m}f\big((1+mn-n)(x_{i}-x_{j})\big)\right.\\\left.-m\sum_{i=1}^{m}f\big((1+mn-n)x_{i}\big)\bigg\|\right\}\\
		+\lim_{n\longrightarrow\infty}\sup\left\{\frac{(m-1)}{m}\bigg\|f\left((1-n)\sum_{i=1}^{m}x_{i}\right)+\sum_{1\leq i<j\leq m}f\big((1-n)(x_{i}-x_{j})\big)\right.\\\left.-m\sum_{i=1}^{m}f((1-n)x_{i})\bigg\|\right\}\\
		+\lim_{n\longrightarrow\infty}\sup\left\{(m-1)\bigg\|f\left(n\sum_{i=1}^{m}x_{i}\right)+\sum_{1\leq i<j\leq m}f\big(n(x_{i}-x_{j})\big)-m\sum_{i=1}^{m}f\big(nx_{i}\big)\bigg\|\right\}\\
		\leq \lim_{n\longrightarrow\infty}\sup \frac{1}{m}\varphi\bigg((1+mn-n)x_{1},\dots,(1+mn-n)x_{m}\bigg)\\
		+\lim_{n\longrightarrow\infty}\sup \frac{(m-1)}{m}\varphi\bigg((1-n)x_{1},\dots,(1-n)x_{m}\bigg)\\
		+\lim_{n\longrightarrow\infty}\sup (m-1)\varphi\bigg(nx_{1},\cdots,nx_{m}\bigg)
		=0,
	\end{multline*}
	for all $x_{1},\cdots,x_{m}\in E$. Which means that the equation \eqref{eqnd} is hyperstable on $E$.
\end{proof}
\begin{corollary}
	Let $\theta$ and $p$ two real numbers such that $\theta\geq0$ and $p<0$. Let $f:E\to Y$ be a mapping  satisfying 
	\begin{equation*}\label{eq1cor1}
		\left\|f\left(\sum_{i=1}^{m}x_{i}\right)+\sum_{1\leq i<j\leq m}f(x_{i}-x_{j})- m\sum_{i=1}^{m}f(x_{i})\right\|\leq \theta \prod_{i=1}^{m}\|x_{i}\|^{p},
	\end{equation*}
	for all $x_{1},\cdots,x_{m}\in X\setminus\{0\}$. Then $f$ satisfies the functional equation \eqref{eqnd} on $E$. 
\end{corollary}
\begin{proof}
	In Theorem \ref{th1-5}, we  suppose that
	$$
	\varphi(x_{1},\cdots,x_{m}):=\theta\prod_{i=1}^{m}\|x_{i}\|^{p}
	$$
	for all $x_{1},\cdots,x_{m}\in E$. We note that
	$$
	\lim_{n\rightarrow\infty}\varphi\big(x,nx,\cdots,nx\big)=\lim_{n\rightarrow\infty}\theta n^{(m-1)p}\|x\|^{p}=0
	$$
	and
	$$
	\lim_{n\rightarrow\infty}\varphi\big(nx,nx,\cdots,nx\big)=\lim_{n\rightarrow\infty}\theta n^{mp}\|x\|^{p}=0,
	$$
	for all $x\in E$, which means that $f$ satisfies the equation \eqref{eqnd} on $E$.
\end{proof}
	
	\section*{Acknowledgments}
	It is our great pleasure to thank the referee their his careful reading of the paper and for several helpful suggestions.
	
	\section*{Contributions}
	All authors contributed significantly to this paper and were involved in drafting and reviewing the manuscript. All authors have read and approved the final version of the manuscript.
	\section*{Ethics declarations}
	
	\subsection*{Availability of data and materials}
	Not applicable.
	\subsection*{Conflict of interest}
	The authors declare that they have no competing interests.
	\subsection*{Funding}
	Not applicable.


\begin{thebibliography}{99}
	
	\bibitem{muaad} M. Almahalebi, A. Charifi, S. Kabbaj, and E. Elqorachi, \emph{A fixed point approach to stability of the quadratic equation}, In: Rassias, T., Tóth, L. (eds) Topics in Mathematical Analysis and Applications. Springer Optimization and Its Applications, \textbf{94} (2014), 53--77. \url{https://doi.org/10.1007/978-3-319-06554-0_3}.
	
	\bibitem{H6} M. Almahalebi and A. Chahbi, \emph{Hyperstability of the Jensen functional equation in ultrametric spaces}, Aequationes Math., \textbf{91} (2017), no. 4, 647--661. \url{https://doi.org/10.1007/s00010-017-0487-6}
	
	\bibitem{H7} M. Almahalebi and A. Chahbi, \emph{Approximate solution of $P$-radical functional equation in 2-Banach spaces}, Acta Math. Sci., \textbf{39} (2019), 551--566. \url{https://doi.org/10.1007/s10473-019-0218-2}
	
	\bibitem{H8} M. Almahalebi, S. AL Ali, M. Hryrou, and Y. Elkettani, \emph{A fixed point theorem in ultrametric $n$-Banach spaces and hyperstability results}, Fixed Point Theory, \textbf{24} (2023), no. 2, 433--458. \url{https://doi.org/10.24193/fpt-ro.2023.2.01}
	
	\bibitem{aoki} T. Aoki, \emph{On the stability of the linear transformation in Banach spaces}, J. Math. Soc. Japan, \textbf{2} (1950), 64--66. \url{https://doi.org/10.2307/2042795}
	
	\bibitem{ref-27} I. Amemiya, \emph{On the representation of complemented modular lattices}, J. Math. Soc. Jpn., \textbf{9} (1957), no. 2, 263--279. \url{https://doi.org/10.2969/jmsj/00920263}
	
	\bibitem{H9} A. Bahyrycz, J. Brzd\c{e}k, E.-S. El-Hady, and Z. Le\'{s}niak, \emph{On Ulam stability of functional equations in 2-normed spaces -- A survey}, Symmetry, \textbf{13} (2021), no. 11, 2200. \url{https://doi.org/10.3390/sym13112200}.
	
	\bibitem{bourgin1} D. G. Bourgin, \emph{Approximately isometric and multiplicative transformations on continuous function rings}, Duke Math. J., \textbf{16} (1949), 385--397. \url{https://doi.org/10.1215/S0012-7094-49-01639-7}
	
	\bibitem{bourgin} D. G. Bourgin, \emph{Classes of transformations and bordering transformations}, Bull. Amer. Math. Soc., \textbf{57} (1951), 223--237.
	
	\bibitem{H1} J. Brzd\c{e}k and K. Ciepli\'{n}ski, \emph{A fixed point approach to the stability of functional equations in non-Archimedean metric spaces}, Nonlinear Anal., \textbf{74} (2011), no. 18, 6861--6867. \url{https://doi.org/10.1016/j.na.2011.06.050}
	
	\bibitem{H2} J. Brzd\c{e}k, J. Chudziak, and Zs. P\'{a}les, \emph{A fixed point approach to stability of functional equations}, Nonlinear Anal., \textbf{74} (2011), no. 17, 6728--6732. \url{https://doi.org/10.1016/j.na.2011.06.052}
	
	\bibitem{H3} J. Brzd\c{e}k, \emph{Stability of additivity and fixed point methods}, Fixed Point Theory Appl., \textbf{1} (2013), 1--9. \url{https://doi.org/10.1186/1687-1812-2013-285}
	
	\bibitem{H4} J. Brzd\c{e}k and K. Ciepli\'{n}ski, \emph{On a fixed point theorem in 2-Banach spaces and some of its applications}, Acta Math. Sci., \textbf{38} (2018), no. 2, 377--744. \url{https://doi.org/10.1016/S0252-9602(18)30755-0}
	
	\bibitem{H5} J. Brzd\c{e}k and K. Ciepli\'{n}ski, \emph{A fixed point theorem in $n$-Banach spaces and Ulam stability}, J. Math. Anal. Appl., \textbf{470} (2019), 632--646. \url{https://doi.org/10.1016/j.jmaa.2018.10.028}
	
%	\bibitem{cheng} S. C. Cheng and J. N. Mordeson, \emph{Fuzzy linear operators and fuzzy normed linear spaces}, Bull. Calc. Math. Soc., \textbf{86} (1994), 429--436.
	
	\bibitem{chol} P. W. Cholewa, \emph{Remarks on the stability of functional equations}, Aequationes Math., \textbf{27} (1984), 76--86. \url{https://doi.org/10.1007/BF02192660}
	
	\bibitem{H11} I.-S. Chang, Y.-H. Lee, and J. Roh, \emph{Representation and stability of general nonic functional equation}, Mathematics, \textbf{11} (2023), 3173. \url{https://doi.org/10.3390/math11143173}.
	
	\bibitem{cz} S. Czerwik, \emph{On the stability of the quadratic mapping in normed spaces}, Abh. Math. Sem. Univ. Hamburg, \textbf{62} (1992), 59--64. \url{https://doi.org/10.1007/BF02941618}
	
	\bibitem{H10} E.-S. El-Hady and J. Brzd\c{e}k, \emph{On Ulam stability of functional equations in 2-normed spaces -- A survey II}, Symmetry, \textbf{14} (2022), no. 7, 1365. \url{https://doi.org/10.3390/sym14071365}.
	
	\bibitem{gav} P. A. G\u{a}vru\c{t}\u{a}, \emph{Generalization of the Hyers-Ulam-Rassias stability of approximately additive mappings}, J. Math. Anal. Appl., \textbf{184} (1994), 431--436. \url{https://doi.org/10.1006/jmaa.1994.1211}
	
	\bibitem{H12} B. Hayati and H. Khodaei, \emph{On triple $\theta$-centralizers}, Int. J. Nonlinear Anal. Appl., \textbf{15} (2024), no. 1, 9--16. \url{https://doi.org/10.22075/ijnaa.2023.30207.4365}
	
	\bibitem{hyers} D. H. Hyers, \emph{On the stability of the linear functional equation}, Proc. Natl. Acad. Sci. USA, \textbf{27} (1941), 222--224. \url{https://doi.org/10.1073/pnas.27.4.222}
	
	\bibitem{H13} H. Khodaei, \emph{Asymptotic behavior of $n$-Jordan homomorphisms}, Mediterr. J. Math., \textbf{17} (2020), 143. \url{https://doi.org/10.1007/s00009-020-01580-w}.

	\bibitem{ref-32} M. Krbec, \emph{Modular interpolation spaces I}, Z. Anal. Anwendungen, (1982), no. 1, 25--40. \url{https://doi.org/10.4171/zaa/3}
	
	\bibitem{maksa} Gy. Maksa and Zs. P\'{a}es, \emph{Hyperstability of a class of linear functional equations}, Acta Math. Acad. Paedag. Ny\'{i}rh\'{a}ziensis, \textbf{17} (2001), 107--112. 
	
	\bibitem{ref-33} L. Maligranda, \emph{Orlicz spaces and interpolation}, Seminários de Matemática, 5, Universidade Estadual de Campinas, Departamento de Matemática, Campinas, Brazil, 1989. \url{https://orcid.org/0000-0002-9584-4083}
	
	\bibitem{ref-26} H. Nakano, \emph{Modulared semi-ordered linear spaces}, Maruzen Co., Ltd., Tokyo, Japan, 1950.
	
	\bibitem{ref-30} W. Orlicz, \emph{Collected papers. Part I, II}, PWN-Polish Scientific Publishers, Warsaw, Poland, 1988.
	
	\bibitem{ras} Th. Rassias, \emph{On the stability of the linear mapping in Banach spaces}, Proc. Amer. Math. Soc., \textbf{72} (1978), 297--300. \url{https://doi.org/10.1090/S0002-9939-1978-0507327-1}.
	
	\bibitem{skof} F. Skof, \emph{Approssimazione di funzioni $\delta$-quadratic su dominio restretto}, Atti. Accad. Sci. Torino Cl. Sci. Fis. Mat. Nat., \textbf{118} (1984), 58--70.
	
	\bibitem{ulam} S. M. Ulam, \emph{Problems in modern mathematics}, Science Editions, John Wiley, New York, 1964.
	
	\bibitem{ref-31} S. Yamamuro, \emph{On conjugate spaces of Nakano spaces}, Trans. Amer. Math. Soc., \textbf{90} (1959), 291--311. \url{https://doi.org/10.2307/1993206}
	
\end{thebibliography}
\end{document}